\documentclass{amsart}
\usepackage[bbgreekl]{mathbbol}
\usepackage{amssymb,amsmath,amsthm}
\usepackage[foot]{amsaddr}
\title{Apartness and the elimination of strong forms of extensionality}
\author{Benno van den Berg$^1$}
\address{${}^1$ ILLC, Universiteit van Amsterdam, P.O. Box 94242, 1090 GE Amsterdam, the Netherlands. E-mail: bennovdberg@gmail.com.}
\date{\today}
\usepackage{ast2}
\usepackage[margin=1.7in]{geometry}

\usepackage{tikz-cd}

\newcommand{\kcom}{\mathsf{k}}
\newcommand{\scom}{\mathsf{s}}
\newcommand{\Rcom}{\mathsf{R}}

\newcommand{\fst}{\mathsf{fst}}
\newcommand{\snd}{\mathsf{snd}}

\newcommand{\pair}{\mathsf{pair}}

\newcommand{\EXT}{\mathsf{EXT}}
\newcommand{\CEXT}{\mathsf{CEXT}}

\newcommand{\usftext}[1]{\textsf{\upshape #1}}

\newcommand{\ha}{\ensuremath{{\usftext{HA}}^\omega}} 
\newcommand{\eha}{\ensuremath{{\usftext{E-HA}}^\omega}} 
\newcommand{\mr}{\ensuremath{{\usftext{mr}}}}

\newcommand{\heha}{\ensuremath{{\usftext{HE-HA}}^\omega}}

\newcommand{\hha}{\ensuremath{{\usftext{H-HA}}^\omega}}

\newcommand{\hazero}{\ensuremath{{\usftext{HA}}_0^\omega}}

\newcommand{\weha}{\ensuremath{{\usftext{WE-HA}}^\omega}}

\newcommand{\ext}{\ensuremath{{\usftext{ext}}}} 
\newcommand{\forallext}{\forall^{\ext{}}\!}
\newcommand{\existsext}{\exists^{\ext{}}\!}

\begin{document}

\begin{abstract}
    We introduce a new version of arithmetic in all finite types which extends the usual versions with primitive notions of extensionality and extensional equality. This new hybrid version allows us to formulate a strong form of extensionality, which we call converse extensionality. Inspired by Brouwer's notion of apartness, we show that converse extensionality can be eliminated in a way which improves on results from our previous work. We also explain how standard proof-theoretic interpretations, like realizability and functional interpretations, can be extended to such hybrid systems, and how that might be relevant to proof-mining.
\end{abstract}

\maketitle

\section{Introduction}

The standard approach towards the semantics of constructive logic is to explain the meaning of a mathematical statement in terms of what counts as evidence in its favour. This idea is the starting point for the Brouwer-Heyting-Kolmogorov interpretation, various realizability interpretations and type theory. 

However, occasionally one also finds the idea of explaining the meaning of a proposition in terms of ``counter-evidence'': a piece of data that serves as evidence \emph{against} a proposition. Examples of this are Brouwer's notion of apartness and G\"odel's Dialectica interpretation.

Brouwer came up with his notion of apartness in order to explain what it means that two real numbers $r$ and $s$ are equal. This is not so easy; however, saying what is needed to show that they are distinct (apart, in Brouwer's terminology) is much easier. Indeed, this one does by providing two non-overlapping intervals with rational endpoints $q_0,q_1,q_2,q_3$ such that $r \in [q_0,q_1]$ and $s \in [q_2,q_3]$. Brouwer then defines equality of two reals as the impossibility of obtaining evidence for their apartness.

With respect to (extensional) equality of functions, we are in a similar situation. It is easy to see what should be a counterargument against their equality, and much less clear how one provides an argument in favour of their equality. Indeed, to give a counterargument against the equality of functions $f$ and $g$ means to point to an input on which they yield distinct outputs. Given this, equality of the functions $f$ and $g$ can then be defined as the impossibility of any counterargument against their equality. 

If one takes the idea of counter-evidence seriously, it is natural to say that for an implication $\varphi \to \psi$ to be valid, not only should evidence for $\varphi$ give rise to evidence for $\psi$ (as in the traditional account), but counter-evidence for $\psi$ should also give rise to counter-evidence for $\varphi$. And this is what happens in G\"odel's Dialectica interpretation. One can see this clearly when one looks at the Dialectica interpretation of the following extensionality principle:
\begin{equation} \label{extprinciple} 
    \begin{aligned}
    \forall f:(\sigma \to 0) \to (\tau \to 0).\forall x,y:\sigma \to 0.\\ \big( \, (\forall u:\sigma.xu =_0 yu) \to (\forall v:\tau.fxv =_0 fyv) \, \big).
    \end{aligned}
\end{equation}
(Here 0 stands for the type of natural numbers.) This says that the functional $f$ preserves extensional equality: if you give it functions $x$ and $y$ as input and these are extensionally equal, then the outputs should also be extensionally equal. Note that counter-evidence against the equality of $x$ and $y$ is a $u:\sigma$ such that $xu \not=_0 yu$ and counter-evidence against the equality of $fx$ and $fy$ is a $v:\tau$ such that $fxv \not=_0 fyv$. So if counter-evidence flows backwards along an implication, it should be possible to find such a $u$ from such a $v$. That is, there should be a functional $Z$ such that
\begin{equation} \label{convextprinciple} 
    \begin{aligned}
    \forall f:(\sigma \to 0) \to (\tau \to 0).\forall x,y:\sigma \to 0.\forall v:\tau. \\ fxv \not=_0 fyv \to x(Zfxyv) \not=_0 y(Zfxyv);
    \end{aligned}
\end{equation}
this is essentially the Dialectica interpretation of the original statement. (Indeed, it is equivalent to it, since $=_0$ is decidable.) In earlier work we referred to (\ref{convextprinciple}) as a \emph{converse extensionality principle}. It is a strong form of extensionality, which implies, but is not implied by, the extensionality principle (\ref{extprinciple}). This means that if we wish to Dialectica interpret extensionality, we are asked to witness converse extensionality, which is not always possible. Indeed, one can find some striking limitative results in the appendix of \cite{Troelstra344}, written by W.A.~Howard.

What is not excluded by these results, however, is that converse extensionality can be eliminated in a manner similar to the well-known elimination of extensionality due to Gandy and Luckhardt (see \cite[Section 10.4]{kohlenbach08}); by this we mean that we can interpret systems with such extensonality principles in systems which do not possess them. Indeed, the main contribution of this paper is to show that that is indeed possible for converse extensionality using precisely Brouwer's notion of apartness: this elimination procedure will be called the \emph{$\alpha$-translation} ($\alpha$ for apartness).

This is not the first time that the we have looked at apartness and converse extensionality. Indeed, the contents of this paper are similar to that of \cite{bergpassmann22}, written together with Robert Passmann. In the conclusion to that paper we suggested that it might be possible to obtain stronger results if one could have non-extensional witnesses for principles like converse extensionality. In this paper we show that that is indeed the case. As a result, we manage to eliminate converse extensionality for all finite types, instead of for very low types only; in addition, we will do this without using continuity principles.

The main obstacle that we needed to overcome in order to have non-extensional witnesses for converse extensionality is that we had to find a good setting where extensional and non-extensional objects can peacefully co-exist. To do this, we will introduce a hybrid system for arithmetic. By this we mean an extension of a neutral system like $\ha$ with primitive notions of extensionality and extensional equality for which we postulate some axioms. Since the principle of converse extensionality is most naturally understood as talking about the interaction between extensional and non-extensional objects, this will be the right setting to formulate it.

The contents of this paper are therefore as follows. We will introduce our hybrid system in Section 2. In Section 3 we will present a useful technical result concerning the extensionality axiom. In Section 4 we will show that Gandy and Luckhardt's elimination of extensionality can be seen as a two-step process with our hybrid system as an intermediate destination. In Section 5 we will present our $\alpha$-translation and show how it eliminates converse extensionality. In Section 6 we will briefly discuss how proof-theoretic interpretations like modified realizability and functional interpretations can be extended to hybrid systems. Finally, we conclude the paper with some discussion and directions for future research in Section 7.

\section{A hybrid system for arithmetic}

In this section we will introduce our hybrid system for arithmetic in finite types. It extends the standard system $\ha$ (as in \cite{TroelstraVanDalen88ii}) with primitive notions of extensionality and extensional equality for which we add some axioms.\footnote{Similar hybrid systems can be built on top of other versions of arithmetic in finite types. We wil be using that there is a notion of equality for all finite for which the congruence laws hold, so, as matter stand, this might not be possible for the system $\weha$, as in \cite{kohlenbach08}. If one still wants to have a version which allows for G\"odel's Dialectica interpretation, an good option is the system $\hazero$ from \cite{ego17}.}

The system $\ha$ is based on multi-sorted intuitionistic logic, where the sorts are the finite types, as generated by the following grammar:
\[ \mbox{type} = 0 \, | \, \mbox{type} \times \mbox{type} \, | \, \mbox{type} \to \mbox{type}. \]
We will use small Greek letters as variables to range over the finite types. The finite type 0 stands for the type of natural numbers and the finite types are closed under product and function types. We use the convention that $\times$ (and $\land$) binds stronger than $\to$ and that $\to$ associates to the right, so $\rho \to \sigma \to \tau$ stands for $\rho \to (\sigma \to \tau)$.

\begin{table} \caption{The combinators} \label{combinators}
    \begin{displaymath}
        \begin{array}{ll}
            \kcom:  &\rho \to \sigma \to \rho \\
            \scom: & (\rho \to \sigma \to \tau) \to (\rho \to \sigma) \to (\rho \to \tau) \\
            \pair: & \sigma \to \tau \to \sigma \times \tau \\
            \fst: &\sigma \times \tau \to \sigma \\
            \snd: & \sigma \times \tau \to \tau \\
            0: & 0 \\
            S: & 0 \to 0 \\
            \Rcom: & \sigma \to (0 \to \sigma \to \sigma) \to (0 \to \sigma)
        \end{array}
    \end{displaymath}
    \end{table}

For all types $\rho, \sigma$ and $\tau$, the language contains the constants that can be found in Table \ref{combinators}, together with their types. If we were being absolutely precise, we should say that for every pair of types $\rho$ and $\sigma$, we have a combinator $\kcom_{\rho,\sigma}$ of type $\rho \to \sigma \to \rho$, and similarly for the other combinators; however, we will never write these indices and the reader is asked to infer these from the context. We will use ${\bf c}$ as a metavariable ranging over all combinators. Terms are built from these combinators and variables using application: that is, if $s$ is a term of type $\sigma \to \tau$ and $t$ as term of type $\sigma$, then the result of applying $s$ to $t$, written as $st$, is a term of type $\tau$. Application associates to the left, and therefore $rst$ stands for $(rs)t$.

The system $\ha$ has a primitive notion of equality at each type $\sigma$, which we will write $\equiv_\sigma$ or $\equiv$ when $\sigma$ is understood. It includes axioms stating that it is an equivalence relation at each type, as well as a congruence:
\begin{displaymath}
    \begin{array}{c}
       x \equiv_\sigma x, \quad x \equiv_\sigma y \to y \equiv_\sigma x, \quad 
       x \equiv_\sigma y \to y \equiv_\sigma z \to x \equiv_\sigma z, \\ x \equiv_{\sigma \to \tau} x' \to y \equiv_\sigma y' \to xy \equiv_{\tau} x'y'
    \end{array}
\end{displaymath}
We will assume the standard axioms for the combinators:
\begin{eqnarray*}
        \kcom xy &\equiv& x \\
        \scom xyz & \equiv & xz(yz) \\
        \fst(\pair \, xy) & \equiv & x \\
        \snd(\pair \, xy) & \equiv & y \\
        \Rcom x y 0 & \equiv & x \\
        \Rcom x y (Sm) & \equiv & ym(\Rcom x y m)
\end{eqnarray*}
We will not use that pairing is surjective (that is, that $\pair(\fst \, x)(\snd \, x) \equiv x$ holds). 

Finally, $\ha$ contains the Peano axioms
\begin{displaymath}
    \begin{array}{c}
        Sx \equiv_0 Sy \to x \equiv_0 y, \quad Sx \not\equiv_0 0,
    \end{array}
\end{displaymath}
as well as the induction axiom for all formulas in the language of $\ha$:
\[ \varphi(0) \to \forall x:0.\big( \, \varphi(x) \to \varphi(Sx) \, \big) \to \forall x:0.\varphi(x). \]
Let us also recall that $\eha$ is the extension of $\ha$ with the following extensionality principles:
\[ x \equiv _{\sigma \times \tau} y \leftrightarrow \fst \, x \equiv_\sigma  \fst \, y \land \snd \, x \equiv_\tau \snd \, y, \quad f \equiv_{\sigma \to \tau} g \leftrightarrow \forall x:\sigma.fx \equiv_\tau gx. \]

The following well-known properties of $\ha$ will be used repeatedly in what follows.

\begin{lemm}{closedtermofanytype}
    For any finite type $\sigma$ there is a closed term $0_\sigma:\sigma$.
\end{lemm}
\begin{proof}
    We define $0_\sigma$ by induction on type structure of $\sigma$:
    \begin{eqnarray*}
        0_0 & :\equiv & 0, \\
        0_{\sigma \times \tau} & :\equiv & \pair \, 0_\sigma \, 0_\tau, \\
        0_{\sigma \to \tau} & :\equiv & \kcom \, 0_\tau.
    \end{eqnarray*}
\end{proof}

\begin{prop}{combinatorycompleteness}
    Let $t$ be a term of type $\tau$ and $x$ be a variable of type $\sigma$. Then there is a term $\lambda x:\sigma.t$ of type $\sigma \to \tau$, whose free variables are those of $t$ minus $x$, such that for any term $s$ of type $\sigma$ we have $\ha \vdash (\lambda x:\sigma.t)s \equiv t[s/x]$.
\end{prop}
\begin{proof}  (See \cite[Lemma 3.15]{kohlenbach08} or \cite[Proposition 9.1.8]{TroelstraVanDalen88ii}) We define $\lambda x:\sigma.t$ by induction on the structure of the term $t$, as follows:
    \begin{eqnarray*}
        \lambda x:\sigma.x & :\equiv & \scom \, \kcom \, \kcom, \\
        \lambda x:\sigma.y & :\equiv & \kcom \, y \quad \mbox{if $y$ is a variable distinct from $x$}, \\
        \lambda x:\sigma.{\bf c} & :\equiv & \kcom \, {\bf c} \quad \mbox{if ${\bf c}$ is one of the combinators from Table \ref{combinators}}, \\
        \lambda x:\sigma.st &:\equiv & \scom \, (\lambda x:\sigma.s) \, (\lambda x:\sigma.t).
    \end{eqnarray*} 
\end{proof}

On top of this neutral theory we build a axiomatic theory of extensionality and extensional equality, which we will denote by $\hha$. To formulate this hybrid version of arithmetic, we extend the language of $\ha$ with a new unary predicate $\ext_\sigma$ and a new binary predicate $=_\sigma$ for each finite type $\sigma$. This means that if $s$ and $t$ are terms of type $\sigma$, then we have new atomic formulas $\ext_\sigma(t)$ for ``$t$ is extensional'' and $s =_\sigma t$ for ``$s$ and $t$ are extensionally equal'', where we will often drop the type symbol $\sigma$. We will use the following abbreviations:
\begin{eqnarray*}
    \forallext x:\sigma.\varphi & :\equiv & \forall x:\sigma.( \, \ext_\sigma(x) \to \varphi \, ) \\
    \existsext x:\sigma.\varphi & :\equiv & \exists x:\sigma.( \, \ext_\sigma(x) \land \varphi \, )
\end{eqnarray*}

\begin{defi}{hha} The system $\hha$ is theory which is formulated in this extended language and which contains, besides the axiom of $\ha$ previously mentioned, the following axioms for all types $\sigma, \tau$:
\begin{align*}
    x =_0 y \leftrightarrow x \equiv_0 y, \quad 
    \forall x:0. \, \ext_0(x), \\ x =_{\sigma \times \tau} y \leftrightarrow \fst \, x =_\sigma \fst \, y \land \snd \, x =_\tau \snd \, y, \quad \ext_{\sigma \times \tau}(x) \leftrightarrow \ext_\sigma(\fst \, x) \land \ext_\tau(\snd \, x) \\
    f =_{\sigma \to \tau} g \leftrightarrow \forallext x:\sigma.fx=_\tau gx, \quad \ext_{\sigma \to \tau}(f) \to \ext_\sigma(x) \to \ext_\tau(fx) \\
    x \equiv_\sigma y \to \ext_\sigma(x) \to \ext_\sigma(y), \quad \ext({\bf c}),
\end{align*}
where $\bf c$ is any of the combinators from Table \ref{combinators}. In addition, $\hha$ has the induction axiom for all formulas in the extended language.

The system $\heha$ is the extension of $\hha$ that we obtain by adding the following \emph{extensionality axiom}:
\begin{align*}
\EXT: \quad \ext_{\sigma \to \tau}(f) \to \ext_\sigma(x) \to \ext_\sigma(y) \to x =_\sigma y \to fx =_\tau fy.
\end{align*}
\end{defi}

\begin{rema}{nonimplications} Despite being very natural, we have decided not to include the axiom $x =_\sigma y \to \ext_\sigma(x) \to \ext_\sigma(y)$. While almost all results in this paper would still have gone through if we had included this axiom, there is one crucial exception, which we will discuss in Subsection 6.2 below.
\end{rema}  

\begin{rema}{redundancyinpresenceofsurjofpairing} If we assume surjectivity of pairing (that is, $\pair(\fst \, x)(\snd \, x) \equiv x$), then the axiom $\ext_{\sigma \times \tau}(x) \leftrightarrow \ext_\sigma(\fst \, x) \land \ext_\tau(\snd \, x)$ is redundant. Indeed, the left-to-right direction follows from the fact that $\fst$ and $\snd$ are extensional in combination with the axiom saying that extensional functions send extensional input to extensional output. The right-to-left direction follows from the same axiom in combination with extensionality of $\pair$ and surjectivity of pairing.
\end{rema}

\begin{rema}{onlyeqoftype0primitive} 
    We will often consider extensional equality $=_\sigma$ as an abbreviation. Indeed, by repeatedly replacing the left-hand side of the equivalences 
    \begin{eqnarray*}
            x =_0 y & \leftrightarrow & x  \equiv_0 y, \\ x =_{\sigma \times \tau} y & \leftrightarrow & \fst \, x =_\sigma \fst \, y \land \snd \, x =_\tau \snd \, y, \\ f =_{\sigma \to \tau} g & \leftrightarrow & \forallext x:\sigma.fx=_\tau gx, 
    \end{eqnarray*}
    with its right-hand side, we can rewrite any formula containing the symbol for extensional equality into an equivalent one that no longer uses that symbol.
\end{rema} 

\begin{lemm}{propertiesofextequalityinhha}
    \begin{enumerate}
        \item[(i)] $\hha \vdash x =_\sigma x$ and, more generally, $\hha \vdash x \equiv y \to x = y$.
        \item[(ii)] $\hha \vdash x =_\sigma y \to y =_\sigma x$
        \item[(iii)] $\hha \vdash x =_\sigma y \to y =_\sigma z \to x =_\sigma z$
    \end{enumerate}   
\end{lemm}
\begin{proof}
    All these statements are proven by induction on the type structure, using the axioms mentioned in \refrema{onlyeqoftype0primitive} above.
\end{proof}

In $\hha$ the $\eta$-axiom and surjectivity of pairing hold with respect to extensional equality, in the following way.

\begin{lemm}{etaandsomesurjpairing}
    $\hha \vdash \forall f:\alpha \to \beta. f = \lambda x:\alpha.fx$ and $\hha \vdash \pair(\fst \,  x)(\snd \, x) = x$.
\end{lemm}

\begin{prop}{closedtermsextensional}
    If $t$ is a term of $\hha$ of type $\sigma$ and the free variables of the term $t$ are $x_1,\ldots,x_n$ of types $\sigma_1,\ldots,\sigma_n$, respectively, then  
    \[ \hha \vdash \ext_{\sigma_1}(x_1) \to \ldots \to  \ext_{\sigma_n}(x_n) \to \ext_\sigma(t). \]
    In particular, we have $\hha \vdash \ext_\sigma(t)$ if $t$ is closed. 
\end{prop}
\begin{proof}
    By induction on the structure of the term $t$.
\end{proof}

\begin{rema}{Elogic}
    Our treatment of the $\ext$-predicate might remind the reader of the $E$-predicate from $E$-logic (see, for instance, \cite[Section 2.2]{TroelstraVanDalen88i}), where it is read as ``existence'' or as ``being defined''. However, there are some striking differences. First of all, we allow quantifiers and free variables to range over non-extensional objects. In addition, we do not insist on ``strictness'': in particular, we do not require $x = y \to \ext(x) \land \ext(y)$. Indeed, every object, even the non-extensional ones, will be extensionally equal to itself (see \reflemm{propertiesofextequalityinhha}(i)). In this respect the current paper differs from \cite{bergpassmann22}, where we did insist on strictness, because that was what matched the categorical semantics in that paper.
\end{rema}    

\section{A closer look at extensionality}

In this section we will take a closer look at the extensionality axiom $\EXT$ in $\hha$, which says that for all finite types $\sigma$ and $\tau$ we have:
\[ \EXT_{\sigma,\tau}: \quad
\ext_{\sigma \to \tau}(f) \to \ext_\sigma(x) \to \ext_\sigma(y) \to x = y \to fx = fy. \]
Our goal in this section is to show that if this axiom holds in $\hha$ for types $\sigma$ and $\tau$ of the form $\sigma' \to 0$ and $\tau' \to 0$, respectively, then it holds for all types $\sigma$ and $\tau$. 

The reason this is useful is that it allows us to reformulate the axiom $\EXT$ in a way which no longer uses the notion of extensional equality. Of course, we have seen in \refrema{onlyeqoftype0primitive} a general method for eliminating extensional equality, but if the type $\sigma$ is very complicated (for instance, if it contains many nested function types), the formula which this produces for $x =_\sigma y$ may also be very complicated. However, the formula $f =_{\tau \to 0} g$ is equivalent to $\forallext x:\tau. fx \equiv_0 gx$, which is rather simple. So it follows from the results in this section that in $\hha$ the axiom $\EXT$ is equivalent to the following principle:
\begin{align*} \EXT': \quad \ext_{(\sigma \to 0) \to (\tau \to 0)}(f) \to \ext_{\sigma \to 0}(x) \to \ext_{\sigma \to 0}(y) \to \\ \forallext u:\sigma. xu \equiv_0 yu \to \forallext v:\tau. fxv \equiv_0 fyv. \end{align*}
The point is that $\EXT'$ is a simple reformulation of $\EXT$ which no longer uses the notion of extensional equality.

For the transparent presentation of our result, we will work in $\hha$ and use a tiny bit of category theory. We will write \ct{C} for the following category:
\begin{description}
    \item[Objects] The objects of \ct{C} are the finite types.
    \item[Morphisms] The morphisms $\alpha \to \beta$ in \ct{C} are equivalence classes of closed terms $t$ of type $\alpha \to \beta$; here we regard $t, t': \alpha \to \beta$ as equivalent if $\hha \vdash t = t'$ (which is equivalent to $\hha \vdash \forallext x:\alpha. tx =_\beta tx'$). 
\end{description}
Note that we always have $\hha \vdash \ext(t)$ for any morphism $t$ in this category, by \refprop{closedtermsextensional}. The identities $\alpha \to \alpha$ are given by $\lambda x:\alpha.x$, while the composition of $s: \alpha \to \beta$ and $t: \beta \to \gamma$ is $\lambda x:\alpha.t(sx)$. The axioms of a category are easily verified. 

A morphism $t: \alpha \to \beta$ in this category will be called \emph{strong} if 
\[ \hha \vdash \ext_\alpha(x) \to \ext_\alpha(y) \to x=_\alpha y \to tx =_\beta ty \]
(note that this independent of the choice of representative, and $\EXT$ implies that every morphism is strong). Clearly, identities are strong and the strong maps are closed under composition, so we obtain a subcategory \ct{S} of \ct{C} having the same objects as \ct{C} and whose morphisms are the strong morphisms of \ct{C}. 

Let us first investigate to which extent these categories are cartesian closed.

\begin{lemm}{categoricalproductcategory}
    The type $\alpha \times \beta$ is the categorical product of the types $\alpha$ and $\beta$ in both \ct{C} and \ct{S}.
\end{lemm}    
\begin{proof}
    The projection maps are $\fst: \alpha \times \beta \to \alpha$ and $\snd: \alpha \times \beta \to \beta$, respectively. The axiom $x =_{\sigma \times \tau} y \leftrightarrow \fst \, x =_\sigma \fst \, y \land \snd \, x =_\tau \snd \, y$ implies that both maps are strong.

    If $s: \gamma \to \alpha$ and $t: \gamma \to \beta$ are morphisms in \ct{C}, then so is \[ (s,t) :\equiv \lambda x:\sigma.\pair(sx)(tx): \gamma \to \alpha \times \beta. \]
    We have $\fst \circ (s,t) = s$ and $\snd \circ (s,t) = t$ as morphisms in \ct{C}. 
    
    It remains to show that $(s,t): \gamma \to \alpha \times \beta$ is the unique morphism with these properties and that $(s,t)$ will be strong as soon as $s$ and $t$ are. However, all of this follows from the axiom $x =_{\sigma \times \tau} y \leftrightarrow \fst \, x =_\sigma \fst \, y \land \snd \, x =_\tau \snd \, y$.
\end{proof}

\begin{lemm}{exponentialcategory}
    The type $\alpha \to \beta$ is the categorical exponential of the types $\alpha$ and $\beta$ in \ct{C} if we assume surjectivity of pairing; if we assume $\EXT_{\alpha,\beta}$, then it is the exponential in \ct{S}.
\end{lemm}
\begin{proof}
    We have a morphism ${\rm ev}: (\alpha \to \beta) \times \alpha \to \beta$ in \ct{C} given by
    \[ \lambda x:(\alpha \to \beta) \times \alpha.(\fst \, x)(\snd \, x). \]
    If $h: \gamma \times \alpha \to \beta$ is a morphism in \ct{C}, then we obtain a morphism $H: \gamma \to (\alpha \to \beta)$ in \ct{C} by:
    \[ \lambda x:\gamma.\lambda y:\alpha.h(\pair \, x \, y). \]
    If we assume surjectivity of pairing or if $h$ is strong, then we have ${\rm ev} \circ (H \times 1_\alpha) = h$. If $H'$ has the same property, then \[ Hxy \equiv {\rm ev} \circ (H \times 1)(\pair \, x \, y) = h(\pair \, x \, y) = {\rm ev} \circ (H' \times 1)(\pair \, x \, y) = H'xy \] for any extensional $x:\gamma$ and extensional $y:\alpha$; hence $H = H'$.
    
    It remains to show that ${\rm ev}$ is strong and that $H$ will be strong as soon as $h$ is. The fact that ${\rm ev}$ is strong follows from the axiom of extensionality $\EXT$ for the types $\alpha$ and $\beta$, as in \reflemm{propertiesofextequalityinheha}.

    Finally, suppose that $h$ is strong and $x, x': \gamma$ are extensional and satisfy $x = x'$. To show that $Hx = Hx'$, we let $y:\alpha$ be an arbitrary extensional element and derive that $Hxy = Hx'y$. Note that $\pair \, x \, y = \pair \, x \, y'$ with $\ext(\pair \, x \, y) = \ext(\pair \, x \, y')$. Because $h$ is strong, this implies that
    \[ Hxy = h(\pair \, x \, y) = h(\pair \, x' \, y) = Hx'y; \]
    we conclude that $H$ is strong as well.
\end{proof}

The following technical definition is our most important tool in this section.

\begin{defi}{strongretract}
Let $\alpha$ and $\beta$ be finite types. We will say that $\alpha$ is a \emph{strong retract} of $\beta$ if there are strong maps $i: \alpha \to \beta$ and $r: \beta \to \alpha$ such that $r \circ i = 1_\alpha$.
\end{defi}

Our next goal is to show that any type $\sigma$ is a strong retract of one of the form $\tau \to 0$. The proof relies on the following lemma.

\begin{lemm}{somepropertiesofstrongretracts}
\begin{enumerate}
    \item[(i)] If $\alpha$ is a strong retract of $\beta$ and $\beta$ is a strong retract of $\gamma$, then $\alpha$ is a strong retract of $\gamma$.
    \item[(ii)] If $\alpha_0$ is a strong retract of $\beta_0$ and $\alpha_1$ of $\beta_1$, then $\alpha_0 \times \alpha_1$ is a strong retract of $\beta_0 \times \beta_1$.
    \item[(iii)] If $\alpha$ is a strong retract of $\beta$, then $\gamma \to \alpha$ is a strong retract of $\gamma \to \beta$. 
    \item[(iv)] If $\alpha$ is strongly isomorphic to $\beta$, then $\gamma \to \alpha$ is strongly isomorphic to $\gamma \to \beta$. 
    \item[(v)] $(\alpha \to \beta) \times (\gamma \to \delta)$ is a strong retract of $\alpha \times \beta \to \gamma \times \delta$.
    \item[(vi)] $\alpha \to \beta \to \gamma$ is a strong retract of $\alpha \times \beta \to \gamma$.
\end{enumerate} 
\end{lemm}
\begin{proof}
    Points (i) and (ii) are simple exercises in category theory, with (ii) using \reflemm{categoricalproductcategory}.

    (iii): If $i: \alpha \to \beta$ and $r: \beta \to \alpha$ are strong maps such that $r \circ i = 1$, then
    \[ i' :\equiv \lambda x:\gamma \to \alpha.\lambda y:\gamma.i(xy) \]
    and 
    \[ r' :\equiv \lambda x:\gamma \to \beta.\lambda y:\gamma.r(xy) \]
    define strong maps $i': (\gamma \to \alpha) \to (\gamma \to \beta)$ and $r': (\gamma \to \beta) \to (\gamma \to \alpha)$ such that $r' \circ i' = 1$. Indeed, if we have extensional $x, x':\gamma$ such that $x = x'$, then $i'x = i'x'$ will follow as soon as we have $i'xy = i'x'y$ for any extensional $y:\gamma$. We have $xy = x'y$ because $x = x'$ and $y$ is extensional, and
    \[ i'xy \equiv i(xy) = i(x'y) \equiv i'x'y, \]
    because $i$ is strong. So $i'$ is strong and a similar argument shows that $r'$ is strong as well.

    To show $r' \circ i' = 1$, we pick extensional $x:\gamma \to \alpha$ and $y: \gamma$ and show that $(r' \circ i')xy = xy$. But we have
    \[ (r' \circ i')xy \equiv r((i'x)y) \equiv r(i(xy)) = xy, \]
    using that $xy$ is extensional and $r \circ i = 1$.

    (iv): By $\alpha$ and $\beta$ being strongly isomorphic, we mean that they are isomorphic in the category \ct{S}. If in the situation as in (iii) above, we also have $i \circ r = 1$, then we must have $i' \circ r' = 1$ as well, by symmetry.

    (v): In the forward directon there is:
    \[ i :\equiv \lambda h:(\alpha \to \beta) \times (\gamma \to \delta).\lambda y:\alpha \times \beta.\pair((\fst \, h)(\fst \, y))((\snd \, h)(\snd \, y)), \]
    while in the backward direction there is:
    \[ r :\equiv \lambda k: \alpha \times \beta \to \gamma \times \delta.\pair(\lambda x:\alpha.\fst(k(\pair \, x \, 0)))(\lambda z:\beta.\snd(k(\pair \, 0 \, z)), \]
    where $0$ is a closed term as in \reflemm{closedtermofanytype}. By arguments similar to those that we have seen before, both operations can be shown to be strong, while $r \circ i = 1$.

    (vi): In the forward directon there is:
    \[ i :\equiv \lambda f:\alpha \to \beta \to \gamma.\lambda x:\alpha \times \beta.f(\fst \, x)(\snd \, x), \]
    while in the backward direction there is:
    \[ r :\equiv \lambda g: \alpha \times \beta \to \gamma.\lambda x:\alpha.\lambda y:\beta.g(\pair \, x \, y). \] 
    Both operations can be shown to be strong, while $r \circ i = 1$. (Note that if we had surjectivity of pairing, we could show $i \circ r = 1$ as well.)
\end{proof}    

\begin{prop}{retractargument}
    Every type $\sigma$ is a strong retract of one of the form $\tau \to 0$.
\end{prop}
\begin{proof}
    We prove this statement by induction on the structure of $\sigma$, using the properties of strong retracts established in the previous lemma. In addition, we use that we can code pairs of natural numbers as a single natural number: in particular, we have a strong isomorphism $0 \times 0 \to 0$.

    For type 0, we have maps $i: 0 \to (0 \to 0)$ and $r:(0 \to 0) \to 0$ given by $\lambda x:0.\lambda y:0.x$ and $\lambda f:0 \to 0.f0$, respectively. Clearly, both are strong morphisms and $r \circ i = 1$.

    If $\sigma_0$ is a strong retract of $\tau_0 \to 0$ and $\sigma_1$ of $\tau_1 \to 0$, then $\sigma_0 \times \sigma_1$ is a strong retract of $(\tau_0 \to 0) \times (\tau_1 \to 0)$, which in turn is a strong retract of $\tau_0 \times \tau_1 \to 0 \times 0$, which is strongly isomorphic to $\tau_0 \times \tau_1 \to 0$.

    Finally, if $\sigma$ is a strong retract of $\tau \to 0$, then $\rho \to \sigma$ is a strong retract of $\rho \to (\tau \to 0)$, which is a strong retract of $(\rho \times \tau) \to 0$.
 \end{proof}

\begin{prop}{simplificationofextensionality}
    In $\hha$, if we have $\EXT_{\tau_0 \to 0,\tau_1 \to 0}$ for all finite types $\tau_0,\tau_1$, then we have $\EXT_{\sigma_0,\sigma_1}$ for all finite types $\sigma_0,\sigma_1$.
\end{prop}    
\begin{proof}
    We reason in $\hha$ and assume we are given a map $f:\sigma_0 \to \sigma_1$. We have to show that $f$ is strong under the assumption that $\EXT_{\tau_0 \to 0,\tau_1 \to 0}$ holds for all finite types $\tau_0,\tau_1$. Using the previous proposition, we can exhibit $\sigma_k$ as a strong retract of $\tau_k$ via strong maps $i_k: \sigma_k \to (\tau_k \to 0)$ and $r_k: (\tau_k \to 0) \to \sigma_k$ such that $r_k \circ i_k = 1$ for $k \in \{0, 1 \}$. Let us write $g :\equiv i_1 \circ f \circ r_0$, as in the diagram below.
    \begin{displaymath}
        \begin{tikzcd}
            \sigma_0 \ar[d, bend left, "{i_0}"] \ar[r, "f"] & \sigma_1 \ar[d, bend left, "{i_1}"] \\
            \tau_0 \to 0 \ar[r, "g"] \ar[u, bend left, "{r_0}"] & \tau_1 \to 0 \ar[u, bend left, "{r_1}"]
        \end{tikzcd}
    \end{displaymath}
    By assumption, $g$ is strong. Because strong maps are closed under composition and $r_1 \circ g \circ i_0 = f$, the map $f$ is strong as well.
\end{proof}    

\begin{coro}{reformulationEXT}
    Over $\hha$ the extensionality axiom $\EXT$ is equivalent to
    \begin{align*} \EXT': \quad \ext_{(\sigma \to 0) \to (\tau \to 0)}(f) \to \ext_{\sigma \to 0}(x) \to \ext_{\sigma \to 0}(y) \to \\ \forallext u:\sigma. xu \equiv_0 yu \to \forallext v:\tau. fxv \equiv_0 fyv. \end{align*}
\end{coro}

The previous corollary is important for us because it implies that the following converse extensionality principle $\CEXT$ implies $\EXT$.
\begin{align*} \CEXT: \quad \exists Z. \, \forall^\ext f:(\sigma \to 0) \to (\tau \to 0). \, \forall^\ext x,y:\sigma \to 0. \, \forall^\ext v:\tau. \\ \big( \, fxv \not\equiv_0 fyv \to \ext(Zfxyv) \land x(Zfxyv)) \not\equiv_0 y(Zfxyv)  \, \big). \end{align*}
Note that this is indeed a converse extensionality principle as in (\ref{convextprinciple}) from the introduction. As discussed there as well, it will be crucial that we do not demand that the functional $Z$ in $\CEXT$ witnessing this strong form of extensionality is extensional itself. 

\begin{coro}{fromstrongtoweakext}
    $\hha \vdash \CEXT \to \EXT'$ and therefore $\hha \vdash \CEXT \to \EXT$.
\end{coro}

\section{Hybrid arithmetic as an intermediate system}

In this section we will show that there are natural interpretations $\eha \to \heha$ and $\heha \to \ha$, respectively. Let us start with the former.

\begin{theo}{formehatoheha} Let $\varphi$ be a formula in the language of $\eha$ with free variables $x_1,\ldots,x_n$ of types $\sigma_1,\ldots,\sigma_n$, respectively, and let $\varphi^*$ be the formula in $\heha$ obtained from $\varphi$ by replacing $\equiv$ with $=$, $\forall x$ with $\forallext x$ and $\exists x$ with $\existsext x$. Then \[ \eha \vdash \varphi \quad \mbox{if and ony if} \quad \heha \vdash \ext_{\sigma_1}(x_1) \to \ldots \to \ext_{\sigma_n}(x_n) \to \varphi^*. \]
\end{theo}
\begin{proof}
    The direction from left to right is shown by induction on the derivation of $\eha \vdash \varphi$, using \refprop{closedtermsextensional} and the lemma below. The direction from right to left follows from the fact that there is an interpretation of $\heha$ into $\eha$ obtained by declaring every element to be extensional and interpreting extensional equality = as equality $\equiv$.
\end{proof}

\begin{lemm}{propertiesofextequalityinheha}
    $\heha \vdash \forallext f,g \, \forallext x, y \, (x = y \land f = g \to fx = gy)$.
 \end{lemm}
 \begin{proof}
     The extensionality axiom $\EXT$ applied to $f$ gives us that $fx = fy$. From $f = g$ we obtain $fy = gy$ and therefore $fx = gy$.
 \end{proof}

\begin{theo}{eliminationofext}
    There is a proof-theoretic interpretation $\heha \to \ha$ which is the identity on $\ha$-formulas. Therefore $\heha$ is a conservative extension of $\ha$.
\end{theo}
\begin{proof}
    The idea behind the interpretation of $\heha$ into $\ha$ is to regard $=_\sigma$ and $\ext_\sigma$ as abbreviations defined by induction on $\sigma$, as follows.
    \begin{eqnarray*}
        \ext_0(x) & :\equiv & \top \\
        x =_0 y & :\equiv & x \equiv_0 y \\ \\
        \ext_{\sigma \times \tau}(x) & :\equiv & \ext_\sigma(\fst \, x) \land \ext_\tau(\snd \, x) \\
        x =_{\sigma \times \tau} y &:\equiv& \fst \, x =_\sigma \fst \, y \land \snd \, x =_\tau \snd \, y \\ \\
        \ext_{\sigma \to \tau}(f) &:\equiv& \forall x:\sigma \, \big( \, \ext_\sigma(x) \to \ext_\tau(fx) \, \big) \land \\
        & & \forall x,y:\sigma \, \big( \, x=_\sigma y \to \ext_\sigma(x) \to \ext_\sigma(y) \to fx=_\tau fy \, \big) \\
        f =_{\sigma \to \tau} g &:\equiv & \forall x:\sigma  \, \big( \, \ext_\sigma(x) \to fx =_\tau gx \, \big)
    \end{eqnarray*}
    To verify the soundness of this interpretation, we only need to check that it soundly interprets the axioms of $\heha$. This is trivial for most of them, the axioms of the form $\ext({\bf c})$ for the various combinators ${\bf c}$ being the exception; we will only discuss this for the $\scom$-combinator, as the argument for the other combinators is similar.

    We start by showing that ${\bf s}xy$ will be extensional whenever $x$ and $y$ are. This involves two things: showing that ${\bf s}xyz$ is extensional whenever $z$ is, and showing that ${\bf s}xyz = {\bf s}xyz'$ whenever $z,z'$ are extensional and $z = z'$.

    The former follows because $xy(xz)$ will be extensional as soon as $x,y,z$ are and ${\bf s}xyz \equiv xy(xz)$. Similarly, if $z = z'$ and $z$ and $z'$ are extensional, then ${\bf s}xyz \equiv xy(xz) = xy(xz') \equiv {\bf s}xyz'$, because $x$ and $xy$ are extensional.

    The next step is to show that ${\bf s}x$ is extensional whenever $x$ is. For that we need to show that ${\bf s}xy = {\bf s}xy'$ whenever $y$ and $y'$ are extensional and $y = y'$; in other words, we should show that ${\bf s}xyz = {\bf s}xy'z$ whenever $z$ extensional. However, we have $xy = xy'$, because $x$ is extensional, and therefore ${\bf s}xyz \equiv xy(xz) = xy'(xz) = {\bf s}xy'z$, because $xz$ is extensional.
    
    It remains to show that ${\bf s}x = {\bf s}x'$ whenever $x = x'$ and $x$ and $x'$ are extensional, which can be done in a way similar to what we have seen before.
\end{proof}

\begin{rema}{interpretingmore}
    Note that this translation also validates $x = y \to \ext(x) \to \ext(y)$, which we can show by induction on the type $\sigma$.
\end{rema}

\begin{rema}{eliminatingext}
    The composed translation $\eha \to \heha \to \ha$ is nothing but Gandy and Luckhardt's elimination of extensionality. So what we have done is factor this translation as a composition of two translations with $\heha$ as the intermediate system.
\end{rema}

\section{The $\alpha$-translation}

The interpretation of $\heha$ into $\ha$ that we presented in the previous section is in many ways the canonical one. In this section, which is the heart of this paper, we will present an alternative: the $\alpha$-translation, inspired by Brouwer's notion of apartness. The main benefit of this translation is that it allows us to eliminate the strong extensionality principle $\CEXT$, that we introduced in Section 3.

\begin{table} 
    \caption{Auxiliary definitions for the $\alpha$-translation} \label{tab:auxdefinitions}
\begin{eqnarray*}
    0^+ & :\equiv & 0 \\
    0^- & :\equiv & 0 \\
    {\rm dom}_0 & :\equiv & \top \\
    {\rm app}_0 & :\equiv & x \not\equiv_0 y \\ \\
    (\sigma \times \tau)^+ & :\equiv & \sigma^+ \times \tau^+ \\
    (\sigma \times \tau)^- &:\equiv & (\sigma^- \times \tau^-) \times 0 \\
    {\rm dom}_{\sigma \times \tau} &:\equiv & {\rm dom}_\sigma(\fst \, x) \land {\rm dom}_\tau(\snd \, x) \\
    {\rm app}_{\sigma \times \tau} &:\equiv & \big( \, \snd \, z \equiv 0 \to {\rm app}_\sigma(\fst \, x, \fst \, y, \fst (\fst \, z)) \, \big) \land \\ 
    & & \big( \, \, \snd \, z \not\equiv 0 \to {\rm app}_\sigma(\snd \, x, \snd \, y, \snd (\, \fst \, z)) \, \big) \\ \\
    (\sigma \to \tau)^+ & :\equiv & (\sigma^+ \to \tau^+) \times (\sigma^+ \to \sigma^+ \to \tau^- \to \sigma^-) \\
    (\sigma \to \tau)^- & :\equiv & \sigma^+ \times \tau^- \\
    {\rm dom}_{\sigma \to \tau} & :\equiv & \forall u:{\sigma^+} ( \, {\rm dom}_\sigma(u) \to {\rm dom}_\tau((\fst \, x)u) \, ) \land 
    \forall u:{\sigma^+}, v:{\sigma^+}, w:{\tau^-} \,  \\ & & \big( \, {\rm dom}_\sigma(u) \to {\rm dom}_\tau(v) \to {\rm app}_\tau((\fst \, x)(u), (\fst \, x)(v),w) \to \\
    & & {\rm app}_\sigma(u,v,(\snd \, x)uvw) \, \big) \\
    {\rm app}_{\sigma \to \tau} &:\equiv& {\rm dom}_\sigma(\fst \, z) \land {\rm app}_\tau((\fst \, x)(\fst \, z), (\fst \, y)(\fst \, z), \snd \, z) 
    \end{eqnarray*}
\end{table}

As a first step towards defining the $\alpha$-translation, we define for each finite type $\sigma$ two finite types $\sigma^+$ and $\sigma^-$, as well as two formulas ${\rm dom}_\sigma$ and ${\rm app}_\sigma$ in the language of $\ha$. The formula  ${\rm dom}_\sigma$ will have one free variable $x$ of type $\sigma^+$ and ${\rm app}_\sigma$ has free variables $x,y,z$ of types $\sigma^+, \sigma^+$ and $\sigma^-$, respectively. Borrowing from the categorical analysis of modified realizability the distinction between potential and actual realizers, we can give the following intuition for these definitions:
\begin{enumerate}
    \item The elements of type $\sigma^+$ stand for elements of type $\sigma$ together with an argument (potential realizer) for their extensionality. 
    \item The predicate ${\rm dom}_\sigma$ holds if that argument is successful (if the potential realizer is an actual realizer).
    \item Elements of type $\sigma^-$ are arguments (potential realizers) for the statement that two elements of type $\sigma$ are apart from each other.
    \item The predicate ${\rm app}_\sigma$ holds if that argument is successful (that potential realizer is an actual realizer).
    \item Crucially, potential realizers always exist, because for any type $\sigma$ there is a closed term $0:\sigma$. Of course, these need not be actual realizers.
\end{enumerate}
The clauses can be found in Table \ref{tab:auxdefinitions}. Note that if $s: (\sigma \to \tau)^+$ and $t: \sigma^+$, then $(\fst \, s)t: \tau^+$. In what follows we will write $s * t$ for $(\fst \, s)t$, so that $s * t: \tau^+$.

\begin{lemm}{basicpropofapartness}
The following statements are provable in $\ha$ for each type $\sigma$:
\begin{enumerate}
    \item There is no $z$ such that ${\rm app}_\sigma(x,x,z)$.
    \item There is a functional $s: \sigma^+ \to \sigma^+ \to \sigma^- \to \sigma^-$, such that if ${\rm app}_\sigma(x,y,z)$, then ${\rm app}_\sigma(y,x,sxyz)$.
    \item There is a functional $t: \sigma^+ \to \sigma^+ \to \sigma^+ \to \sigma^- \to \sigma^- \times 0$, which given $x, y, z:\sigma^+$ and $u:\sigma^-$ such that ${\rm app}_\sigma(x,y,u)$, computes an element $t = txyzu$ such that ${\rm app}_\sigma(x,z,\fst \, t)$ whenever $\snd \, t \equiv_0 0$ and ${\rm app}_\sigma(y,z,\snd \, t)$ whenever $\snd \, t \not\equiv_0 0$.
\end{enumerate}    
\end{lemm}

\begin{prop}{alphatranslationcombinators}
    For each combinator ${\bf c}:\sigma$ as in Table \ref{combinators} there is a combinator ${\bf c}^\alpha: \sigma^+$ with $\ha \vdash {\rm dom}({\bf c}^\alpha)$, which, provably in $\ha$, satisfies the defining equation for that combinator with respect to $*$. For instance, for each pair of finite types $\sigma, \tau$ there is a combinator ${\bf k}^\alpha:(\sigma \to \tau \to \sigma)^+$ such that $\ha \vdash {\rm dom}({\bf k}^\alpha)$ and $\ha \vdash {\bf k}^\alpha*x*y \equiv x$.
\end{prop}
\begin{proof}
    Let us treat the combinators $\fst,\snd,\pair$ first. We need a combinator
    \[ \fst^\alpha: (\sigma \times \tau \to \tau)^+ = (\sigma^+ \times \tau^+ \to \sigma^+) \times (\sigma \times \tau \to \sigma \times \tau \to \sigma^- \to (\sigma^- \times \tau^-) \times 0). \]
    We put $\fst^\alpha :\equiv \pair \, \fst \,  (\lambda x,x':\sigma \times \tau.\lambda e:\sigma^-.\pair(\pair \, e \, 0) \, 0))$. Note what the second component does: given $x, x':\sigma \times \tau$ and a potential realizer $e$ for the apartness of $\fst^\alpha * x$ and $\fst^\alpha * x'$, it computes a potential realizer of the apartness of $x$ and $x'$; this is done in such a way that if $e$ is an actual realizer for the apartness of $\fst^\alpha * x$ and $\fst^\alpha * x'$, then the potential realizer thus obtained is an actual realizer for the apartness of $x$ and $y$. In a similar fashion, we put $\snd^\alpha :\equiv \pair \, \snd \,  (\lambda x,x':\sigma \times \tau.\lambda e:\tau^-.\pair(\pair \, 0 \, e) \, (S0))$.

    Next, we need a combinator $\pair^\alpha$ having type $(\sigma \to \tau \to \sigma \times \tau)^+$, which is
    \begin{align*} \big( \, \sigma^+ \to (\tau^+ \to \sigma^+ \times \tau^+) \times (\tau^+ \to \tau^+ \to (\sigma^- \times \tau^-) \times 0 \to \tau^-) \, \big) \times \\ \big( \, \sigma^+ \to \sigma^+ \to (\tau^+ \times ((\sigma^- \times \tau^-) \times 0)) \to \sigma^- \, \big) \end{align*}
    This really consists of three parts: from $x:\sigma^+, y:\tau^+$ we should compute an element of type $(\sigma \times \tau)^+ = \sigma^+ \times \tau^+$, for which we take $\pair \, x \, y$. Secondly, for any $x:\sigma^+,y,y':\tau^+$ and a potential realizer $e$ of the apartness of $\pair \, x \, y$ and $\pair \, x \, y'$ we compute a potential realizer of the apartness of $y$ and $y'$. For that we choose $\snd \, (\fst \, e)$. This has the additional property that if $e$ was an actual realizer, then so is $\snd \, (\fst \, e)$; to see this, note that there is never an actual realizer for the apartness of $x$ with itself, by item (1) from \reflemm{basicpropofapartness}. Thirdly, for any $x,x':\sigma^+, y: \tau^+$ and potential realizer $e$ of the apartness of $\pair \, x \, y$ and $\pair \, x' \, y$ we need to compute a potential realizer of the apartness of $x$ and $x'$: for that we take $\fst( \, \fst \, e)$. Again, this will be actual realizer whenever $e$ is.

    We see here the pattern that repeats itself for the other combinators: in each case, we define a certain function $f(x_1,\ldots,x_n)$. In addition, we should show that if we have a value $f(x_1,\ldots,x_n)$ and we change the input on one entry, say we change $x_i$ to $x_i'$, then from the inputs and a potential realizer of the apartness of the outputs $f(x_1,\ldots,x_n)$ and $f(x_1,\ldots,x_{i-1},x_i',x_{i+1}, \ldots x_n)$, we are able to compute a potential realizer of the apartness of $x_i$ and $x_i'$; this needs to be done in such a way that if the potential realizer of the apartness of the outputs was an actual realizer, then so is the potential realizer of the apartness of $x_i$ and $x_i'$ that we compute.

    To define the combinator ${\bf k}^\alpha$ we therefore need to supply three functions. First of all, given $x: \sigma^+$ and $y: \tau^+$, we output $f(x,y) :\equiv x$. Secondly, given $x,x',y$ and a potential realizer $e$ of the apartness of $f(x,y) = x$ and $f(x',y) = x'$, we output $e$ as a potential realizer of the apartness of $x$ and $x'$. Thirdly, given $x,y,y'$ and a potential realizer of the apartness of $f(x,y) = x$ and $f(x,y') = x$, we output $0$ as a potential realizer of the apartness of $y$ and $y'$. (Note that $e$ will never be an actual realizer, again by item (1) from \reflemm{basicpropofapartness}.)

    To define the combinator ${\bf s}^\alpha$, we need to supply four functions. 
    \begin{enumerate}
        \item[(i)] Given $x: (\rho \to \sigma \to \tau)^+$, $y: (\rho \to \sigma)^+$ and $z: \rho^+$, we have $f(x,y,z) = x * z * (y * z): \tau^+$. 
        \item[(ii)] If $e: \tau^-$ is a potential realizer of the apartness of $f(x,y,z)$ and $f(x',y,z)$, then $\pair \, z \, (\pair \, (y * z) \, e): \rho^+ \times (\sigma^+ \times \tau^-)$ is a potential realizer of the apartness of $x$ and $x'$. 
        \item[(iii)] Suppose $e:\tau^-$ is a potential realizer of the apartness of $f(x,y,z)$ and $f(x,y',z)$. Note that from $x$ we can extract a function $a: \rho^+ \to \sigma^+ \to \sigma^+ \to \tau^- \to \sigma^-$ such that if $d$ is a potential realizer of the apartness of $x * u * v$ and $x * u * v'$, then $auvv'd$ is a potential realizer of the apartness of $v$ and $v'$. Therefore $az(y * z)(y' * z)$ is a potential realizer of the apartness of $y * z$ and $y' * z$, which implies that $\pair \, z \, (az(y * z)(y' * z))$ is a potential realizer of the apartness of $y$ and $y'$.
        \item[(iv)] Suppose we are given a potential realizer of the apartness of $f(x,y,z)$ and $f(x,y,z')$. Let us first of all note that from elements $f,f':(\sigma \to \tau)^+$ and $u,u':\sigma^+$ and a potential realizer of the apartness of $f * u$ and $f' * u'$ we can compute a potential realizer of the apartness of $f$ and $f'$ or a potential realizer of the apartness of $u$ and $u'$; indeed, using item (iii) from the previous lemma we can compute from a potential realizer of the apartness of $f * u$ and $f' * u'$ a potential realizer of the apartness of $f * u$ and $f' * u$ or a potential realizer of the apartness of $f' * u$ and $f' * u'$. In the former case the pair consisting of $u$ and that potential realizer is a potential realizer of the apartness of $f$ and $f'$; in the latter case we can use the second component of $f'$ to give us a potential realizer of the apartness of $u$ and $u'$. In the case at hand this means that we find either a potential realizer of the apartness of $x * z$ and $x * z'$ or a potential realizer of the apartness of $y * z$ and $y * z'$. In the first case we use the second component of $x$ and in the second case the second component of $y$ to give us a potential realizer of the apartness of $z$ and $z'$.
    \end{enumerate}

    For the arithmetical part, we need terms $0^\alpha:0^+ = 0$ as well as $S^\alpha: (0 \to 0)^+ = (0 \to 0) \to (0 \to 0 \to 0 \to 0)$, for which we choose $0^\alpha :\equiv 0$ and $S^\alpha :\equiv \pair \, S \, (\lambda x,y,z:0.z)$. Finally, we need a term $\Rcom^\alpha: (\sigma \to (0 \to \sigma \to \sigma) \to 0 \to \sigma))^+$, for which we need to supply four functions:
    \begin{enumerate}
        \item[(i)] Given $x:\sigma^+$, $y:(0 \to \sigma \to \sigma)^+$ and $n: 0$, we need an element $f(x,y,n)$ of type $\sigma^+$, for which we choose $f(x,y,n) :\equiv \Rcom \, x \, (\lambda m:0.\lambda z:\sigma^+.y*m*z)  \, n$.
        \item[(ii)] Next, suppose we are given a potential realizer of the apartness of $f(x,y,n)$ and $f(x',y,n)$. If $n \equiv 0$, then this also a potential realizer of teh apartness of $x$ and $x'$. If $f \equiv Sm$, then we can use the second component of $y$ to compute from this a potential realizer of the apartness of $f(x,y,m)$ and $f(x',y,m)$. This procedure we repeat until we are in the first case. 
        \item[(iii)] Next, suppose we are given a potential realizer of the apartness of $f(x,y,n)$ and $f(x,y',n)$. If $n \equiv 0$, then this is a potential realizer of the apartness of $x$ with itself; such an element is never an actual realizer, so when we can simply choose $0$ as the potential realizer of the apartness of $y$ and $y'$. If $n \equiv Sm$, then we are given a potential realizer of the apartness of $y * m * f(x,y,m)$ and $y' * m * f(x,y',m)$. Using the remark from (iv) above, we obtain from this (1) a potential realizer of the apartness of $y$ and $y'$, or (2) a potential realizer of the apartness of $m$ with itself, or (3) a potential realizer of the apartness of $f(x,y,m)$ and $f(x,y',m)$. In the first case we are finished; in the second case we output 0 as a potential realizer of the apartness of $y$ and $y'$; in the third case, we repeat the algorithm for $m$.
        \item[(iv)] Finally, suppose we are given a potential realizer of the apartness of $f(x,y,n)$ and $f(x,y,n')$. We need a potential realizer of the apartness of $n$ and $n'$, for which we simply choose 0.
    \end{enumerate}
    The remaining verifications are not too difficult and left to the reader.
\end{proof}    

In what follows we assume that we have fixed the terms ${\bf c}^\alpha$ as in the proof of the previous proposition. We wish to extend this to a translation $t^\alpha$ for any term $t$. In order to do so, we will assume in addition that we have assigned to each variable $x$ of type $\sigma$ a variable $x^\alpha$ of type $\sigma^+$ in such a way that for any pair of distinct variables $x, y$ of type $\sigma$ the variables $x^\alpha$ and $y^\alpha$ are distinct as well.

\begin{defi}{alphatranslationterms} To each term $t$ of $\heha$ of type $\sigma$ with free variables $x_1,\ldots,x_n$ of types $\sigma_1,\ldots,\sigma_n$, respectively, we associate a term $t^\alpha$ of $\ha$ of type $\sigma^+$ with free variables $x_1^\alpha,\ldots,x_n^\alpha$ of types $\sigma^+_n,\ldots,\sigma^+_n$, respectively, as follows:
\begin{eqnarray*}
    x^\alpha & :\equiv & x^\alpha, \\
    {\bf c}^\alpha & :\equiv & {\bf c}^\alpha, \\
    (s \, t)^\alpha & :\equiv & s^\alpha * t^\alpha.
\end{eqnarray*}
\end{defi}
  
\begin{definition}
    To each formula $\varphi$ of $\heha$ with free variables $x_1,\ldots,x_n$ of types $\sigma_1,\ldots,\sigma_n$, respectively, we associate a formula $\varphi^\alpha$ of $\ha$ with free variables $x_1^\alpha,\ldots,x_n^\alpha$ of types $\sigma^+_n,\ldots,\sigma^+_n$, respectively, as follows:
\begin{eqnarray*}
    \big( \, \ext_\sigma(t) \, \big)^\alpha &:\equiv & {\rm dom}_\sigma(t^\alpha) \\
    \big( \, s \equiv_\sigma t \, \big)^\alpha & :\equiv & s^\alpha \equiv_{\sigma^+} t^\alpha \\
    \big( \, s =_\sigma t \, \big)^\alpha& :\equiv & \lnot \exists x:{\sigma^-}. \, {\rm app}(s^\alpha, t^\alpha, x) \\
    \bot^\alpha & :\equiv & \bot \\
    \big( \, \varphi \Box \psi \, \big)^\alpha & :\equiv & \varphi^\alpha \Box \psi^\alpha \mbox{ for } \Box \in \{ \lor, \land, \to \} \\
    \big( \, \exists x:\sigma. \, \varphi \, \big)^\alpha & :\equiv & \exists x:{\sigma^+}. \, \varphi^\alpha \\
    \big( \, \forall x:\sigma. \, \varphi \, \big)^\alpha & :\equiv & \forall x:{\sigma^+}. \, \varphi^\alpha
\end{eqnarray*}

\end{definition}

\begin{theorem} If $\heha \vdash \varphi$, then $\ha \vdash \varphi^\alpha$.
\end{theorem}
\begin{proof}
    By induction on the derivation of $\heha \vdash \varphi$.
\end{proof}

\begin{theo}{alphawitnessesconvext} $\ha \vdash \CEXT^\alpha$. 
\end{theo}
\begin{proof}
    Recall that
    \begin{align*} \CEXT: \quad \exists Z. \, \forall^\ext f:(\sigma \to 0) \to (\tau \to 0). \, \forall^\ext x,y:\sigma \to 0. \, \forall^\ext v:\tau. \\ \big( \, fxv \not\equiv_0 fyv \to \ext(Zfxyv) \land x(Zfxyv)) \not\equiv_0 y(Zfxyv)  \, \big). \end{align*}
    Therefore we would like to have \[ Z * f * x * y * v \equiv \fst((\snd \, f)xy(\pair \, v \, 0)). \] 
    We can achieve this by putting
    \[ Z :\equiv \pair(\lambda f.\pair(\lambda x.\pair(\lambda y.\pair(\lambda v.\fst((\snd \, f)xy(\pair \, v \, 0))) \, 0 ) \, 0) \, 0) \, 0. \]
\end{proof}

\begin{rema}{alphatranslationmore}
    By induction on the type structure of $\sigma$, one can also show that the $\alpha$-translation interprets the principle $x =_\sigma y \to \ext_\sigma(x) \to \ext_\sigma(y)$.
\end{rema}

\begin{rema}{comparisonwithearlierpaper}
    The $\alpha$-translation that we presented in this section is a variation on the $\alpha$-translation from the appendix to our paper \cite{bergpassmann22}. A minor difference is that the clauses there are a bit more complicated, because they have been designed in such a way that ${\rm app}(x,y,z)$ implies ${\rm dom}(x)$ and ${\rm dom}(y)$. We have dropped this strictness requirement here (see \refrema{Elogic}).

    The crucial difference is that in the current paper the $\alpha$-translation is an interpretation of the hybrid system $\heha$, instead of $\eha$. Indeed, without a hybrid system, $\CEXT$ and \reftheo{alphawitnessesconvext} cannot even be formulated. Of course, $\eha$ can be embedded into $\heha$, as we showed in \reftheo{formehatoheha}, but the principle $\CEXT$ lies outside the image of this embedding. 
\end{rema}    

\section{Realizability and functional interpretations for hybrid arithmetic}

In this section we will indicate how interpretations like modified realizability and the Diller-Nahm functional interpretation can be extended to $\hha$. We will not discuss these matters in too much detail; also, we will assume familiarity with the standard treatment of these interpretations for $\ha$ (for which, see \cite{Troelstra344}, \cite{TroelstraVanDalen88ii} or \cite{kohlenbach08}).

\subsection{Modified realizability} Kreisel's modified realizability interpretation can be seen as an interpretation of $\ha$ into itself. This can be extended to an interpretation of $\hha$ into itself by adding the following clauses:
\begin{eqnarray*}
    x \, \mr \, \ext(y) & :\equiv & \ext(y), \\
    x \, \mr \, y = z & :\equiv & y = z.
\end{eqnarray*}
This interpret also validates principles like $\EXT$ or $\CEXT$, assuming they hold in the interpreting system: in particular, $\mr$ can also be extended to an interpretation of $\heha$ into itself.

\subsection{A functional interpretation} It is also possible to set up a functional interpretation of $\hha$ into itself. For this we choose the version of $\hha$ without a primitive notion of extensional equality, using \refrema{onlyeqoftype0primitive}. In that case we only have $\ext$ as a new symbol, which we will interpret as itself. For our functional interpretation, we choose the Diller-Nahm variant of G\"odel's original Dialectica interpretation \cite{dillernahm74}: indeed, we have to, because $\ext$ is an atomic formula which is not decidable.\footnote{Not to mention that we have written this paper with $\ha$ as our base system, for which G\"odel's original interpretation does not work, because for higher types $\sigma$ the equality predicate $\equiv_\sigma$ cannot be shown to be decidable in $\ha$. Indeed, for G\"odel's Dialectica interpretation to work we need all atomic formulas to be decidable; the benefit of the Diller-Nahm variant is that it also works when this is not the case.}

It is not hard to see that in this way the Diller-Nahm variant also works as an interpretation of $\hha$ into itself. Note, however, that this interpretation need not interpret the formula
\[ x = y \to \ext(x) \to \ext(y). \]
This has been our main reason for excluding this axiom from $\hha$.

Let us what happens with extensionality, for which we take the formulation
\begin{align*} \EXT': \quad \ext_{(\sigma \to 0) \to (\tau \to 0)}(f) \to \ext_{\sigma \to 0}(x) \to \ext_{\sigma \to 0}(y) \to \\ \forallext u:\sigma. xu \equiv_0 yu \to \forallext v:\tau. fxv \equiv_0 fyv, \end{align*}
as in \refcoro{reformulationEXT}. The Diller-Nahm functional interpretation asks us in this case to find a functional $E$, which given $f: (\sigma \to 0) \to (\tau \to 0),x, y:\sigma \to 0$ and $v:\tau$ computes a finite list $Efxyv$ such that if $v$ is extensional and $xu \equiv_0 xu$ for every extensional $u$ in the list $Efxyv$, then $fxv \equiv_0 fyv$. If our interpreting system includes
\begin{align*} \CEXT: \quad \exists Z. \, \forall^\ext f:(\sigma \to 0) \to (\tau \to 0). \, \forall^\ext x,y:\sigma \to 0. \, \forall^\ext v:\tau. \\ \big( \, fxv \not\equiv_0 fyv \to \ext(Zfxyv) \land x(Zfxyv)) \not\equiv_0 y(Zfxyv)  \, \big), \end{align*}
then we can take $Efxyv$ to be the list consisting only of $Zfxyv$. Indeed, $\CEXT$ is self-interpreting, in that if we assume it in the interpreting system, then we can use that to show that the Diller-Nahm interpretation validates that principle as well. So, in short, we have a Diller-Nahm interpretation of $\hha + \CEXT$ into itself, while $\hha \vdash \CEXT \to \EXT$, as in \refcoro{fromstrongtoweakext}; therefore we can also interpret $\heha$ into $\heha + \CEXT$ using the Diller-Nahm interpretation.

\section{Conclusion}

We have introduced hybrid systems for arithmetic, $\hha$ and $\heha$, and shown how it can be used to formulate a strong extensionality principle, which we have dubbed converse extensionality. Inspired by Brouwer's notion of apartness, we have defined a proof-theoretic interpretation, which we have called the $\alpha$-translation, which can be used to eliminate this converse extensionality principle.

One potential application of this work is in proof-mining \cite{kohlenbach08}, in particular, in proof-mining arguments using extensionality principles. One standard way to treat these is by first eliminating them using the elimination of extensionality and then applying functional interpretation. This would correspond to going from $\eha$ to $\ha$ along the top of the following (non-commuting) diagram, combining elimination of extensionality (EE) and the (Diller-Nahm) functional interpretation (DN).
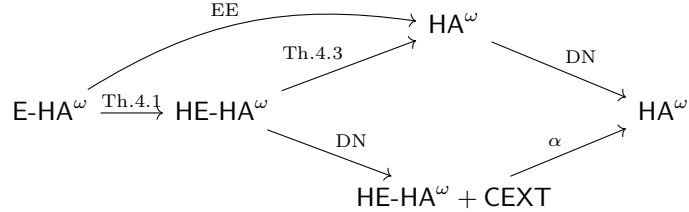
\begin{displaymath}
    \begin{tikzcd} 
        & & \ha \ar[dr, "{\rm DN}"] \\
        \eha \ar[r, "{\rm Th. 4.1}"] \ar[urr, "{\rm EE}", bend left = 20] & \heha \ar[ur, "{\rm Th. 4.3}"] \ar[dr, "{\rm DN}"] & & \ha \\
        & & \heha + \CEXT \ar[ur, "\alpha"]
    \end{tikzcd}
\end{displaymath}
The $\alpha$-translation provides an alternative to this, by allowing us to go along the bottom of the diagram above. Since the $\alpha$-translation of equality statements looks easier than their functional interpretation, this deserves further exploration.

In addition, the idea of a hybrid system to formulate and study hybrid principles should be more widely applicable. For instance, there are discontinuous functionals realizing continuity principles (\cite[Theorem 2.6.7]{Troelstra344} or \cite{troelstra77}), and one could also imagine non-majorizable functionals realizing majorizability statements. We believe this is another promising direction for future research.

\bibliographystyle{plain} \bibliography{bibliography}

\end{document}